\documentclass{amsart}

\usepackage{amsmath,amsfonts,amsthm,amssymb,amscd}
\usepackage{appendix}
\usepackage{cite}
\usepackage{geometry}
\usepackage{graphicx}
\usepackage[colorlinks,linkcolor=blue,citecolor=red]{hyperref}
\hypersetup{CJKbookmarks}
\usepackage{setspace}
\usepackage{spverbatim}
\usepackage{url}
\newtheorem{remark}{Remark}
\newtheorem{theorem}{Theorem}[section]

\newtheorem{conjecture}{Conjecture}[section]
\newtheorem{proposition}{Proposition}[section]

\newtheorem{question}{Question}

\numberwithin{equation}{section}

\author{Chengyang Shao}
\title{Longtime Dynamics of Irrotational Spherical Water Drops: Initial Notes}

\begin{document}
\maketitle
\begin{spacing}{1.2}
In this note, we propose several unsolved problems concerning the irrotational oscillation of a water droplet under zero gravity. We will derive the governing equation of this physical model, and convert it to a quasilinear dispersive partial differential equation defined on the sphere, which formally resembles the capillary water waves equation but describes oscillation defined on curved manifold instead. Three types of unsolved mathematical problems related to this model will be discussed in observation of hydrodynamical experiments under zero gravity\footnote{There are numbers of visual materials on such experiments conducted by astronauts. See for example \url{https://www.youtube.com/watch?v=H_qPWZbxFl8&t} or \url{https://www.youtube.com/watch?v=e6Faq1AmISI&t}.}: (1) Strichartz type inequalities for the linearized problem (2) existence of periodic solutons (3) normal form reduction and generic lifespan estimate. It is pointed out that all of these problems are closely related to certain Diophantine equations, especially the third one.

\section{Capillary Spherical Water Waves Equation: Derivation}

\subsection{Water Waves Equation for a Bounded Water Drop} 
Comparing to gravity water waves problems, the governing equation for a spherical droplet of water under \emph{zero gravity} takes a very different form. At a first glance it looks similar to the water waves systems as mentioned above, but some crucial differences do arise after careful analysis. To the author's knowledge, besides those dealing with generic free-boundary Euler equation (\cite{CoSh2007}, \cite{ShZe2008}), the only reference on this problem is Beyer-G\"{u}nther \cite{BeGu1998}, in which the local well-posedness of the equation is proved using a Nash-Moser type implicit function theorem. We will briefly discribe known results for gravity water waves problems in the next subsection.

To start with, let us pose the following assumptions on the fluid motion that we try to describe:
\begin{itemize}
    \item (A1) The perfect, irrotational fluid of constant density $\rho_0$ occupies a smooth, compact region in $\mathbb{R}^3$.
    \item (A2) There is no gravity or any other external force in presence.
    \item (A3) The air-fluid interface is governed by the Young-Laplace law, and the effect of air flow is neglected.
\end{itemize}

We assume that the boundary of the fluid region has the topological type of a smooth compact orientable surface $M$, and is described by a time-dependent embedding $\iota(t,\cdot):M\to\mathbb{R}^3$. We will denote a point on $M$ by $x$, the image of $M$ under $\iota(t,\cdot)$ by $M_t$, and the region enclosed by $M_t$ by $\Omega_t$. The outer normal will be denoted by $N(\iota)$. We also write $\bar\nabla$ for the flat connection on $\mathbb{R}^3$.

Adopting assumption (A3), we have the Young-Laplace equation:
$$
\sigma_0 H(\iota)=p_i-p_e,
$$
where $H(\iota)$ is the (scalar) mean curvature of the embedding, $\sigma_0$ is the surface tension coefficient (which is assumed to be a constant), and $p_i,p_e$ are respectively the inner and exterior air pressure at the boundary; they are scalar functions on the boundary and we assume that $p_e$ is a constant. Under assumptions (A1) and (A2), we obtain Bernoulli's equation, sometimes referred as the pressure balance condition, on the evolving surface:
\begin{equation}\label{BernEq}
\left.\frac{\partial\Phi}{\partial t}\right|_{M_t}+\frac{1}{2}|\bar\nabla\Phi|_{M_t}|^2-p_e=-\frac{\sigma_0}{\rho_0}H(\iota),
\end{equation}
where $\Phi$ is the velocity potential of the velocity field of the air. Note that $\Phi$ is determined up to a function in $t$, so we shall leave the constant $p_e$ around for convenience reason that will be explained shortly. According to assumption (A1), the function $\Phi$ is a harmonic function within the region $\Omega_t$, so it is uniquely determined by its boundary value, and the velocity field within $\Omega_t$ is $\bar\nabla\Phi$. The kinematic equation on the free boundary $M_t$ is naturally obtained as
\begin{equation}\label{VelEq}
\frac{\partial\iota}{\partial t}\cdot N(\iota)
=\bar\nabla\Phi|_{M_t}\cdot N(\iota).
\end{equation}

Finally, we would like to discuss the conservation laws for (\ref{BernEq})-(\ref{VelEq}). The preservation of volume $\text{Vol}(\Omega_t)=\text{Vol}(\Omega_0)$ is a consequence of incompressibility. The system describes an Eulerian flow without any external force, so the center of mass moves at a uniform speed along a fixed direction, i.e.
\begin{equation}\label{CenterofMass}
\frac{1}{\mathrm{Vol}(\Omega_0)}\int_{\Omega_t}Pd\mathrm{Vol}(P)=V_0t+C_0,
\end{equation}
with Vol being the Lebesgue measure, $P$ marking points in $\mathbb{R}^3$, $V_0$ and $C_0$ being the velocity and starting position of center of mass respectively. Furthermore, the total momentum is conserved, and since the flow is a potential incompressible one, the conservation of total momentum is expressed as
\begin{equation}\label{ConsMomentum}
\int_{M_t}\rho_0\Phi N(\iota)d\mathrm{Area}(M_t)\equiv\rho_0\mathrm{Vol}(\Omega_0) V_0.
\end{equation}
Most importantly, it is not surprising that (\ref{BernEq})-(\ref{VelEq}) is a Hamilton system (the Zakharov formulation for water waves; see Zakharov \cite{Zakharov1968}), with Hamiltonian
\begin{equation}\label{Hamiltonian}
\sigma_0\mathrm{Area}(\iota)+\frac{1}{2}\int_{\Omega_t}\rho_0|\bar\nabla\Phi|^2d\mathrm{Vol}
=\sigma_0\mathrm{Area}(M_t)+\frac{1}{2}\int_{M_t}\rho_0\Phi|_{M_t}\left(\bar\nabla\Phi|_{M_t}\cdot N(\iota)\right) d\mathrm{Area},
\end{equation}
i.e. potential proportional to surface area plus kinetic energy of the fluid.

\subsection{Converting to a Differential System}
It is not hard to verify that the system (\ref{BernEq})-(\ref{VelEq}) is invariant if $\iota$ is composed with a diffeomorphism of $M$; we may thus regard it as a \emph{geometric flow}. If we are only interested in perturbation near a given configuration, we may reduce system (\ref{BernEq})-(\ref{VelEq}) to a non-degenerate dispersive differential system concerning two scalar functions defined on $M$, just as Beyer and G\"{u}nther did in \cite{BeGu1998}. In fact, during a short time of evolution, the interface can be represented as the graph of a function defined on the initial surface: if $\iota_0:M\to\mathbb{R}^3$ is a fixed embedding close to the initial embedding $\iota(0,x)$, we may assume that $\iota(t,x)=\iota_0(x)+\zeta(t,x)N_0(x)$, where $\zeta$ is a scalar ``height" function defined on $M_0$ and $N_0$ is the outer normal vector field of $M_0$.

\begin{figure}[h]
\centering
\includegraphics[width=0.4\textwidth,angle=0]{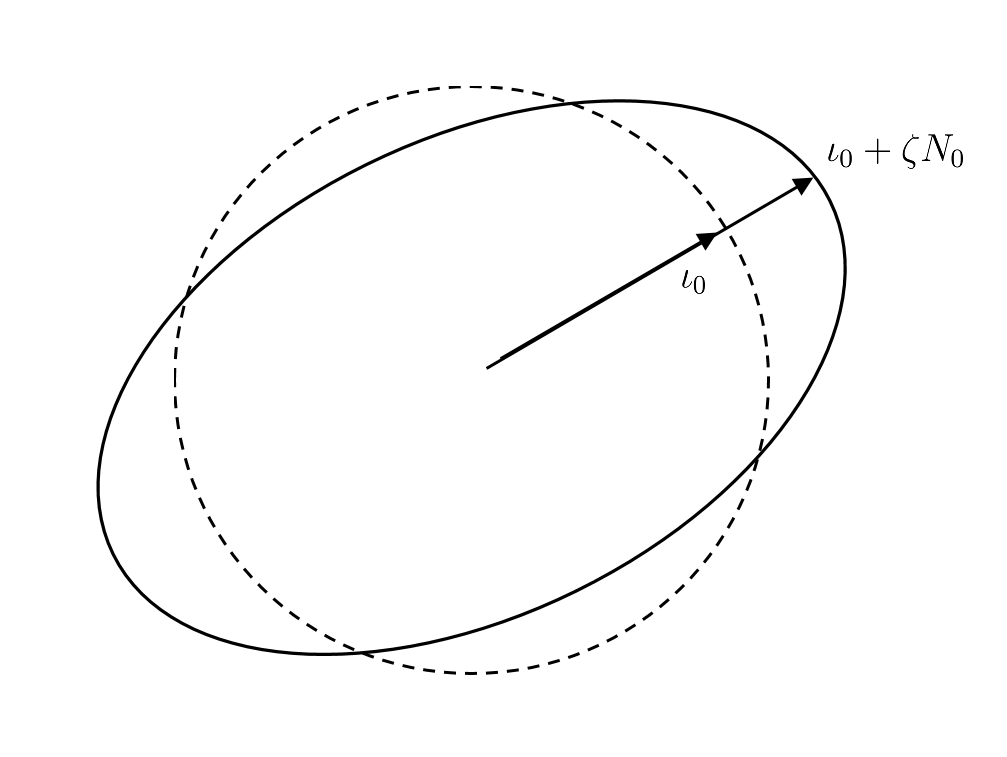}
\caption{The shape of the surface}
\end{figure}

With this observation, we shall transform the system (\ref{BernEq})\&(\ref{VelEq}) into a non-local system of two real scalar functions $(\zeta,\phi)$ defined on $M$, where $\zeta$ is the ``height" function described as above, and $\phi(t,x)=\Phi(t,\iota(t,x))$ is the boundary value of the velocity potential, pulled back to the underlying manifold $M$.

The operator
$$
B_\zeta:\phi\to\bar\nabla\Phi|_{M_t}
$$
maps the pulled-back Dirichlet boundary value $\phi$ to the boundary value of the gradient of $\Phi$. We shall write
$$
(D[\zeta]\phi)N(\iota)
$$
for its normal part, where $D[\zeta]$ is the Dirichlet-Neumann operator corresponding to the region enclosed by the image of $\iota_0+\zeta N_0$. Thus
$$
\frac{\partial\zeta}{\partial t}N_0\cdot N(\iota)=D[\zeta]\phi.
$$

We also need to calculate the restriction of $\partial_t\Phi$ on $M_t$ in terms of $\phi$ and $\iota$. By the chain rule,
$$
\begin{aligned}
\left.\frac{\partial\Phi}{\partial t}\right|_{M_t}
&=\frac{\partial\phi}{\partial t}-\bar\nabla\Phi|_{M_t}\cdot\frac{\partial\iota}{\partial t}\\
&=\frac{\partial\phi}{\partial t}-\left(\bar\nabla\Phi|_{M_t}\cdot N_0\right)\frac{\partial\zeta}{\partial t}\\
&=\frac{\partial\phi}{\partial t}-\frac{1}{N_0\cdot N(\iota)}\left(\bar\nabla\Phi|_{M_t}\cdot N_0\right)\cdot D[\zeta]\phi.
\end{aligned}
$$
We thus arrive at the following nonlinear system:
\begin{equation}\label{EQ(M)}\tag{EQ(M)}
\left\{
\begin{aligned}
\frac{\partial\zeta}{\partial t}
&=\frac{1}{N_0\cdot N(\iota)}D[\zeta]\phi,\\
\frac{\partial\phi}{\partial t}
&=\frac{1}{N_0\cdot N(\iota)}\left(B_\zeta\phi\cdot N_0\right)\cdot D[\zeta]\phi-\frac{1}{2}|B_\zeta\phi|^2-\frac{\sigma_0}{\rho_0}H(\iota)+p_e,
\end{aligned}
\right.
\end{equation}
where $\iota=\iota_0+\zeta N_0$. 

\begin{remark}\emph{
We may obtain an explicit expression of $B_\zeta\phi=\bar\nabla\Phi|_{M_t}$ in terms of $\phi$ (together with the connection $\nabla_0$ on the fixed embedding $\iota_0(M)$), just as standard references did for Euclidean or periodic water waves, but that is not necessary for our discussion at the moment. It is important to keep in mind that the preservation of volume and conservation of total momentum (\ref{CenterofMass})-(\ref{ConsMomentum}) convert to integral equalities of $(\zeta,\phi)$. These additional restrictions are not obvious from the differential equations (\ref{EQ(M)}), though they can be deduced from (\ref{EQ(M)}) since they are just rephrase of the original physical laws (\ref{BernEq})-(\ref{VelEq}). 
}
\end{remark}

For $M=S^2$, the case that we shall discuss in detail, we refer to the system as the \emph{capillary spherical water waves equation}. To simplify our discussion, we shall be working under the center of mass frame, and require the eigenmode $\Pi^{(0)}\phi$ to vanish for all $t$. This could be easily accomplished by absorbing the eigenmode into $\phi$ since the equation is invariant by a shift of $\phi$. In a word, from now on, we will be focusing on the non-dimensional capillary spherical water waves equation
\begin{equation}\label{EQ}\tag{EQ}
\left\{
\begin{aligned}
\frac{\partial\zeta}{\partial t}
&=\frac{1}{N_0\cdot N(\iota)}D[\zeta]\phi,\\
\frac{\partial\phi}{\partial t}
&=\frac{1}{N_0\cdot N(\iota)}\left(B_\zeta\phi\cdot N_0\right)\cdot D[\zeta]\phi-\frac{1}{2}|B_\zeta\phi|^2-H(\iota),
\end{aligned}
\right.
\end{equation}
where $\iota=(1+\zeta)\iota_0$, and $\Pi^{(0)}\phi\equiv0$. We assume that the total volume of the fluid is $4\pi/3$, so that the preservation of volume is expressed as
\begin{equation}\label{Vol}
\frac{1}{3}\int_{S^2}(1+\zeta)^3d\mu_0\equiv\frac{4\pi}{3},
\end{equation}
where $\mu_0$ is the standard measure on $S^2$. The inertial movement of center of mass (\ref{CenterofMass}) and conservation of total momentum (\ref{ConsMomentum}) under our center of mass frame are expressed respectively as
\begin{equation}\label{SphereCons.}
\int_{S^2}(1+\zeta)^4N_0d\mu_0=0,
\quad
\int_{S^2}\phi N(\iota)d\mu(\iota)=0,
\end{equation}
where $\mu(\iota)$ is the induced surface measure. Further, the Hamiltonian of the system is
\begin{equation}\label{SHamilton}
\mathbf{H}[\zeta,\phi]=\mathrm{Area}(\iota)+\frac{1}{2}\int_{S^2}\phi\cdot D[\zeta]\phi \cdot d\mu(\iota),
\end{equation}
and for a solution $(\zeta,\phi)$ there holds $\mathbf{H}[\zeta,\phi]\equiv4\pi$.

Up to this point, we are still working within the realm of well-established frameworks. We already know that the general free-boundary Euler equation is locally well-posed due to the work of \cite{CoSh2007} or \cite{ShZe2008}, and due to the curl equation the curl free condition persists during the evolution. On the other hand, the Cauchy problem of system (\ref{BernEq}) and (\ref{VelEq}) is known to be locally well-posed, due to Beyer and G\"{u}nther in \cite{BeGu1998}. They used an iteration argument very similar to a Nash-Moser type argument in the sense that it involves multiple scales of Banach spaces and ``tame" maps in a certain sense. Finally, it is not hard to transplant the potential-theoretic argument of Wu \cite{Wu1999} to prove the local well-posedness. 

To sum up, we already know that the system (\ref{EQ(M)}) for a compact orientable surface $M$ (hence (\ref{EQ}) specifically) is locally well-posed. But this is all we can assert for the motion of a water droplet under zero gravity. In the following part of this note, we will propose several questions and conjectures concerning the long-time behaviour of water droplets under zero gravity.

\subsection{Previous Works on Water Waves}
There has already been several different approaches to describe the motion of perfect fluid with free boundary. One is to consider the motion of perfect fluid occupying an arbitrary domain in $\mathbb{R}^2$ or $\mathbb{R}^3$, either with or without surface tension. The motion is described by a free boundary value problem of Euler equation. This generic approach was employed by Countand-Shkoller \cite{CoSh2007} and Shatah-Zeng \cite{ShZe2008}. Both groups proved the local-wellposedness of the problem. This approach has the advantage of being very general, applicable to all geometric shapes of the fluid. 

On the other hand, when coming to potential flows of a perfect fluid, the curl free property results in \emph{dispersive} nature of the problem. The motion of a curl free perfect fluid under gravity and a free boundary value condition is usually referred to as \emph{the gravity water waves problem}. The first breakthroughs in understanding local well-posedness were works of Wu \cite{Wu1997} \cite{Wu1999}, who proved local well-posedness of the gravity water waves equation without any smallness assumption. Lannes \cite{Lannes2005} extended this to more generic bottom shapes. Taking surface tension into account, the problem becomes \emph{gravity-capillary water waves}. Schweizer \cite{Sch2005} proved local well-posedness with small Cauchy data of the gravity-capillary water waves problem, and Ming-Zhang \cite{MingZhang2009} proved local well-posedness without smallness assumption. Alazard-Metevier \cite{AM2009} and Alazard-Burq-Zuily \cite{ABZ2011} used para-differential calculus to obtain the optimal regularity for local well-posedness of the water waves equation, either with or without surface tension.

For discussion of long time behavior, it is important to take into account the dispersive nature of the problem. For linear dispersive properties, there has been work of Christianson-Hur-Staffilani \cite{CHS2010}. For gravity water waves living in $\mathbb{R}^2$, works on lifespan estimate include Wu \cite{Wu2009almost} and Hunter-Ifrim-Tataru \cite{HIT2016} (almost global result), Ionescu-Pusateri \cite{IP2015} and Alazard-Delort \cite{AD2015} and Ifrim-Tataru \cite{IT2014} (global result). For gravity water waves living in $\mathbb{R}^3$, there are works of Germain-Masmoudi-Shatah \cite{GMS2012} and Wu \cite{Wu2011} (no surface tension), Germain-Masmoudi-Shatah \cite{GMS2015} (no gravity), Deng-Ionescu-Pausader-Pusateri \cite{DIPP2017} (gravity-capillary water waves) and Wang \cite{Wang2020} (gravity-capillary water waves with finite depth). These results all employed different forms of decay estimates derived from dispersive properties. 

As for long time behavior of periodic water waves, Berti-Delort \cite{BD2018} considered gravity-capillary water waves defined on $\mathbb{T}^1$, Berti-Feola-Pusateri \cite{BFP2018} considered gravity water waves defined on $\mathbb{T}^1$, Ionescu-Pusateri \cite{IoPu2019} considered  gravity-capillary water waves defined on $\mathbb{T}^2$, and obtained an estimate on the lifespan beyond standard energy method. All of the three groups used para-differential calculus and suitable normal form reduction; the results of Berti-Delort and Ionescu-Pusateri were proved for physical data of full Lebesgue measure.

To sum up, all results on the gravity water waves problem listed above are concerned with an equation for two scalar functions defined on a fixed flat manifold, being one of the following: $\mathbb{R}^1$, $\mathbb{R}^2$, $\mathbb{T}^1$, $\mathbb{T}^2$, sometimes called the ``bottom" of the fluid. These two functions represents the geometry of the liquid-gas interface and the boundary value of the velocity potential, respectively. The manifold itself is considered as the bottom of the container in which all dynamics are performed. We observe that the differential equation (\ref{EQ(M)}) is, mathematically, fundamentally different from the water waves equations that have been well-studied.

\section{Initial Notes on Unsolved Problems}
In this section, we propose unsolved problems related to the spherical capillary water waves system (\ref{EQ(M)}) with $M=S^2$. Not surprisingly, these problems all have deep backgrounds in number theory.

\subsection{Linearization Around the Static Solution}
We are mostly interested in the stability of the static solution of (\ref{EQ(M)}). A static solution should be a fluid region whose shape stays still, with motion being a mere shift within the space. In this case, we have $p_e=0$ since the reference is relatively static with respect to the air. Moreover, the velocity field $\bar\nabla\Phi$ and ``potential of acceleration" $\partial\Phi/\partial t$ must both be spatially uniform, so that the left-hand-side of the pressure balance condition (\ref{BernEq}) is a function of $t$ alone. It follows that $\iota(M)$ is always a compact embedded surface of constant mean curvature, hence in fact always an Euclidean sphere by the Alexandrov sphere theorem (see \cite{MP2019}), and we may just take $M=S^2$. Moreover, since $\iota(S^2)$ should enclose a constant volume by incompressibility, the radius of that sphere does not change. After suitable scaling, we may assume that the radius is always 1, and $\rho_0=1$, $\sigma_0=1$, to make (\ref{BernEq})-(\ref{VelEq}) non-dimensional. Finally, by choosing the center of mass frame, we may simply assume that the spatial shift is always zero, so that the velocity potential $\Phi\equiv\Phi_0$, a real constant. It is harmless to fix it to be zero. 

Thus, under our convention, a static solution of (\ref{BernEq})-(\ref{VelEq}) takes the form
\begin{equation}\label{Stat}
\left(
\begin{array}{c}
\iota(t,x)\\
\Phi(t,x)
\end{array}
\right)=
\left(
\begin{array}{c}
\iota_0(x)\\
\displaystyle{
0
}
\end{array}
\right),
\end{equation}
where $a\in\mathbb{R}^3$ is a constant vector, and $\iota_0$ is the standard embedding of $S^2$ as $\partial B(0,1)\subset\mathbb{R}^3$. Equivalently, this means that a static solution of (\ref{EQ(M)}) under our convention must be $(\zeta,\phi)=(0,0)$. Note here that the Gauss map of $\iota_0$ coincides with itself. 

We can now start our perturbation analysis around a static solution at the linear level. Let $\mathcal{E}^{(n)}$ be the space of spherical harmonics of order $n$, normalized according to the standard surface measure on $S^2$. In particular, $\mathcal{E}^{(1)}$ is spanned by three components of $N_0$. Let $\Pi^{(n)}$ be the orthogonal projection on $L^2(S^2)$ onto $\mathcal{E}^{(n)}$, $\Pi_{\leq n}$ be the orthogonal projection on $L^2(S^2)$ onto $\bigoplus_{k\leq n}\mathcal{E}^{(k)}$, $\Pi_{\geq n}$ be the orthogonal projection on $L^2(S^2)$ onto $\bigoplus_{k\geq n}\mathcal{E}^{(k)}$. For $\iota=(1+\zeta)\iota_0$, the linearization of $-H(\iota)$ around the sphere $\zeta\equiv0$ is $\Delta \zeta+2\zeta$, where $\Delta$ is the Laplacian on the sphere $S^2$; cf. the standard formula for the second variation of area in \cite{BD2012}. Then $H'(\iota_0)$ acts on $\mathcal{E}^{(n)}$ as the multiplier $-(n-1)(n+2)$. Note that even if we consider the dimensional form of (\ref{EQ}), there will only be an additional scaling factor $\sigma_0/(\rho_0R^2)$, where $R$ is the radius of the sphere. On the other hand, the following solution formula for the Dirichlet problem on $B(0,1)$ is well-known: if $f\in L^2(S^2)$, then the harmonic function in $B(0,1)$ with Dirichlet boundary value $f$ is determined by
$$
u(r,\omega)=\sum_{n\geq0}r^n(\Pi^{(n)}f)(\omega),
$$
where $(r,\omega)$ is the spherical coordinate in $\mathbb{R}^3$. Thus the Dirichlet-Neumann operator $D[\iota_0]$ acts on $\mathcal{E}^{(n)}$ as the multiplier $n$. Note again that even if we consider the dimensional form (\ref{EQ}), there will only be an additional scaling factor $R^{-1}$. 

Thus, setting
$$
u=\Pi^{(0)}\zeta+\Pi^{(1)}\zeta
+\sum_{n\geq2}\sqrt{(n-1)(n+2)}\cdot\Pi^{(n)}\zeta+i\sum_{n\geq1}\sqrt{n}\cdot\Pi^{(n)}\phi,
$$
we find that the linearization of (\ref{EQ}) around the static solution (\ref{Stat}) is a linear dispersive equation
\begin{equation}\label{EQLin}
\frac{\partial u}{\partial t}+i\Lambda u=0,
\end{equation}
where the $3/2$-order elliptic operator $\Lambda$ is given by a multiplier
$$
\Lambda=\sum_{n\geq2}\sqrt{n(n-1)(n+2)}\cdot \Pi^{(n)}=:\sum_{n\geq0}\Lambda(n)\Pi^{(n)}.
$$
Note that $(\zeta,\phi)$ is completely determined by $u$. At the linear level, there must hold $\Pi^{(0)}u\equiv0$ because the first variation of volume must be zero; and $\Pi^{(1)}u\equiv0$ because of the conservation laws (\ref{SphereCons.}).

Let us also re-write the original nonlinear system (\ref{EQ}) into a form that better illustrates its perturbative nature. For simplicity, we use $O(u^{\otimes k})$ to abbreviate a quantity that can be controlled by $k$-linear expressions in $u$, and disregard its continuity properties for the moment. For example, $\|u\|_{H^1}^2+\|u\|_{H^2}^4$ is an expression of order $O(u^{\otimes2})$ when $u\to0$.

Since the operator $\Lambda$ acts degenerately on $\mathcal{E}^{(0)}\oplus \mathcal{E}^{(1)}$, we should be more careful about the eigenmodes $\Pi^{(0)}u$ and $\Pi^{(1)}u$. The volume preservation equation (\ref{Vol}) implies $\partial_t\Pi^{(0)}\zeta=O(u^{\otimes2})$. Projecting (\ref{EQ}) to $\mathcal{E}^{(1)}$, which is spanned by the components of $N_0$, we obtain $\partial_t\Pi^{(1)}\zeta=\Pi^{(1)}\phi+O(u^{\otimes2})=O(u^{\otimes2})$ since the conservation law (\ref{SphereCons.}) implies $\Pi^{(1)}\phi=O(u^{\otimes2})$; and $\partial_t\Pi^{(1)}\phi=O(u^{\otimes2})$ since $H'(\iota_0)=-\Delta-2$ annihilates $\mathcal{E}^{(1)}$. We can thus formally re-write the nonlinear system (\ref{EQ}) as the following:
\begin{equation}\label{EQFormal}
\frac{\partial u}{\partial t}+i\Lambda u=\mathfrak{N}(u),
\end{equation}
with $\mathfrak{N}(u)=O(u^{\otimes2})$ vanishing quadratically as $u\to0$. Note that we are disregarding all regularity problems at the moment.

\subsection{Question at Linear Level}

At the linear level, our first unanswered question is 
\begin{question}\label{Q1}
Does the solution of the linear capillary spherical water waves equation (\ref{EQLin}) satisfy a Strichartz type estimate of the form
$$
\|e^{it\Lambda}f\|_{L^p_TL^q_x}\lesssim_T\|f\|_{H^s},
$$
where $L^p_TL^q_x=L^p([0,T];L^q(S^2))$, and the admissible indices $(p,q)$ and $s$ should be determined?
\end{question}

Answer to Question \ref{Q1} should be important in understanding the \emph{dispersive} nature of linear capillary spherical water waves. For Schr\"{o}dinger equation on a compact manifold, a widly cited result was obtained by Burq-Gérard-Tzvetkov \cite{BGT2004}: 
\begin{theorem}
On a general compact Riemannian manifold $(M^d,g)$, there holds
$$
\|e^{it\Delta_g}f\|_{L^p_tL^q_x([0,T]\times M)}\lesssim_T\|f\|_{H^{1/p}(M)},
$$
where 
$$
\frac{2}{p}+\frac{d}{q}=\frac{d}{2},
\quad
p>2.
$$
\end{theorem}
The authors used a time-localization argument for the parametrix of $\partial_t-i\Delta_g$ to prove this result. For the sphere $S^d$, this inequality is not optimal. The authors further used a Bourgain space argument to obtain the optimal Strichartz inequality:
\begin{theorem}
Let $(S^d,g)$ be the standard n-dimensional sphere. For a function $f\in C^\infty(S^d)$, there holds the Strichartz inequality
$$
\|e^{it\Delta_g}f\|_{L^p_tL^q_x([0,T]\times M)}\lesssim_T\|f\|_{H^{s}(M)},\quad s>s_0(d)
$$
where 
$$
s_0(2)=\frac{1}{8},\quad s_0(d)=\frac{d}{4}-\frac{1}{2},\,d\geq3.
$$
Furthermore these inequalities are optimal in the sense that the Sobolev index $s$ cannot be less than or equal to $s_0(d)$.
\end{theorem}
The proof is the consequence of two propositions. The first one is the ``decoupling inequality on compact manifolds", in particular the following result proved by Sogge \cite{Sogge1988}:
\begin{proposition}
Let $\Pi_k$ be the spectral projection to eigenspaces with eigenvalues in $[k^2,(k+1)^2]$ on $S^d$. Then there holds
$$
\|\Pi_k\|_{L^2\to L^q}
\leq C_qn^{s(q)},
$$
where
$$
s(q)=\left\{
\begin{matrix}
\frac{d-1}{2}\left(\frac{1}{2}-\frac{1}{2q}\right),\quad & 2\leq q\leq\frac{2(d+1)}{d-1}\\
\frac{d-1}{2}-\frac{d}{q},\quad & \frac{2(d+1)}{d-1}\leq q\leq\infty.
\end{matrix}
\right.
$$
These estimates are sharp in the following sense: if $h_k$ is a zonal spherical harmonic function of degree $k$ on $S^d$, then as $k\to\infty$,
$$
\|h_k\|_{L^q}\simeq C_qk^{s(q)}\|h_k\|_{L^2}.
$$
\end{proposition}
The second one is a Bourgain space embedding result:
\begin{proposition}
For a function $f\in C^\infty_0(\mathbb{R}\times S^d)$, define the Bourgain space norm
$$
\|f(t,x)\|_{X^{s,b}}
:=\left\|\langle\partial_t+i\Delta_g\rangle^bf(t,x)\right\|_{L^2_tH^s_x}.
$$
Then for $b>1/2$ and $s>s_0(d)$, there holds
$$
\|f\|_{L^4(\mathbb{R}\times S^d)}\leq C_{s,b}\|f\|_{X^{s,b}}.
$$
\end{proposition}
The key ingredient for proving this proposition is the following number-theoretic result:
$$
\#\{(p,q)\in\mathbb{N}^2:p^2+q^2=A\}=O(A^\varepsilon).
$$
As for optimality of the Strichartz inequality, the authors of \cite{BGT2004} implemented standard results of Gauss sums.

The parametrix and Bourgain space argument can be repeated without essential change for the linear capillary spherical water waves equation (\ref{EQLin}), but this time the Bourgain space argument would be more complicated: the Bourgain space norm is now
$$
\|f(t,x)\|_{X^{s,b}}
:=\left\|\langle\partial_t+i\Lambda\rangle^bf(t,x)\right\|_{L^2_tH^s_x},
$$
and the embedding result becomes
$$
\|f\|_{L^4(\mathbb{R}\times S^2)}
\leq C_{s,b}\|f\|_{X^{s,b}}
$$
for $f\in C^\infty_0(\mathbb{R}\times S^2)$ and all $s>3\rho/8+1/8$, $b>1/2$, where $\rho$ is the infimum of all exponents $\rho'$ such that when $A\to\infty$, the number
$$
\#\left\{
(n_1,n_2)\in\mathbb{N}^2:\,\frac{1}{2}\leq\frac{n_2}{n_1}\leq2,\,
|\Lambda(n_1)+\Lambda(n_2)-A|\leq\frac{1}{2}
\right\}
\leq C_{\rho'}A^{\rho'}.
$$
Some basic analytic number theory implies $\rho=1/3$, and thus the range of $s$ is $s>1/4$. Surprisingly this index is not better than that predicted by the parametrix method. It remains unknown whether this index could be further optimized.

To close this subsection, we note that the capillary spherical water wave lives on a compact region, so the dispersion does not take away energy from a locality to infinity. This is a crucial difference between waves on compact regions and waves in Euclidean spaces. In particular, we do not expect decay estimate for $e^{it\Lambda}f$. For the nonlinear problem (\ref{EQ}), techniques like vector field method (Klainerman-Sobolev type inequalities) do not apply.

\subsection{Rotationally Symmetric Solutions: Bifurcation Analysis}

Illuminated by observations in hydrodynamical experiments under zero gravity, and suggested by the existence of standing gravity capillary water waves due to Alazard-Baldi \cite{AB2015}, we propose the following conjecture:

\begin{conjecture}
There is a Cantor family of small amplitude periodic solutions to the spherical capillary water waves system (\ref{EQ}).
\end{conjecture}

Let us conduct the bifurcation analysis that suggests why this conjecture should be true. By time rescaling, we aim to find solution $(\zeta,\phi,\omega_0)$ of the following system that is $2\pi$-peiodic in $t$:
\begin{equation}\label{Periodic}
\left\{
\begin{aligned}
\omega_0\frac{\partial\zeta}{\partial t}
&=\frac{1}{N_0\cdot N(\iota)}D[\zeta]\phi,\\
\omega_0\frac{\partial\phi}{\partial t}
&=\frac{1}{N_0\cdot N(\iota)}\left(B_\zeta\phi\cdot N_0\right)\cdot D[\zeta]\phi-\frac{1}{2}|B_\zeta\phi|^2-H(\iota),
\end{aligned}
\right.
\end{equation}
together with the conservation laws (\ref{Vol})-(\ref{SHamilton}). Here we refer $\omega_0>0$ as the \emph{fundamental frequency}. The linearization of this system at the equilibrium $(\zeta,\phi)=(0,0)$ is
\begin{equation}
L_{\omega_0} \left(\begin{matrix}
\zeta \\
\phi
\end{matrix}\right):=
\left(\begin{matrix}
\omega_0\partial_t & -D[0] \\
-\Delta-2 & \omega_0\partial_t
\end{matrix}\right)
\left(\begin{matrix}
\zeta \\
\phi
\end{matrix}\right)=0,
\quad 
\Pi^{(0)}\zeta=\Pi^{(1)}\zeta=\Pi^{(1)}\phi=0.
\end{equation}
We restrict to \emph{rotationally symmetric} solutions of the system: that is, water droplets which are always rotationally symmetric with a fixed axis. In addition, we require $\zeta$ to be even in $t$ and $\phi$ to be odd in $t$. The solution thus should take the form
$$
\zeta(t,x)=\sum_{j,n\geq0}\zeta_{jn}\cos(jt) Y_n(x),
\quad
\phi(t,x)=\sum_{j\geq1,n\geq0}\phi_{jn}\sin(jt) Y_n(x),
$$
where $Y_n$ is the $n$'th zonal spherical harmonic, i.e. the (unique) normalized spherical harmonic of degree $n$ that is axially symmetric. In spherical coordinates this means that $Y_n(\theta,\varphi)=P_n(\cos\theta)$, where $P_n$ is the $n$'th Legendre polynomial. Since $\phi_{0n}$ are irrelevant we fix them to be 0. Then
$$
L_{\omega_0} \left(\begin{matrix}
\zeta \\
\phi
\end{matrix}\right)
=\sum_{j,n\geq0}\left(\begin{matrix}
(-\omega_0 j\zeta_{jn}-n\phi_{jn})\sin(jt)Y_n(x) \\
\left((n-1)(n+2)\zeta_{jn}+\omega_0 j\phi_{jn}\right)\cos(jt)Y_n(x)
\end{matrix}\right).
$$
In order that $(\zeta,\phi)^\mathrm{T}\in\mathrm{Ker}L_{\omega_0}$, at the level $n=0$, we must have $\zeta_{j0}=\phi_{j0}=0$ for all $j\geq0$. At the level $n=1$, we have $\zeta_{01}=\phi_{01}=0$, and for $j\geq1$ there holds $\omega_0 j\zeta_{j1}-\phi_{j1}=0$ and $\omega_0 j\phi_{j1}=0$, so $\zeta_{j1}=\phi_{j1}=0$ for all $j\geq0$. Hence $\zeta_{jn},\phi_{jn}$ can be nonzero only for $j\geq1$ and $n\geq2$.

Consequently, $L_{\omega_0}$ has a one-dimensional kernel if and only if the Diophantine equation
\begin{equation}\label{EllipticCurve}
\omega_0^2j^2=n(n-1)(n+2),\quad j\geq1,\,n\geq2
\end{equation}
has exactly one solution $(j_0,n_0)$. If $\omega_0$ has this property, then at the linear level, the lowest frequency of oscillation is
$$
\omega_0 j_0=\sqrt{n_0(n_0-1)(n_0+2)}.
$$

We look into this Diophantine equation. Obviously $\omega_0^2$ has to be a rational number. The equation is closely related to a family of elliptic curves over $\mathbb{Q}$:
$$
E_c:y^2=x(x-c)(x+2c)=x^3+c^2x^2-2c^2x,
\quad c\in\mathbb{N}.
$$
If we set $a/b=\omega_0^2$ (irreducible fraction), then integral solutions of (\ref{EllipticCurve}) are in 1-1 correspondence with integral points with natural number coordinates on the elliptic curve $E_{ab}$, under the following map:
$$
(j_0,n_0)\to(abn_0,a^2bj_0)\in E_{ab}.
$$
Thus we just need to find natural numbers $a,b$ such that there is a unique (up to negation) integral point $(x,y)\in E_{ab}$, where $x>0$ is divided by $ab$, and $y$ is divided by $a^2b$. We seek for $n_0$ as small as possible with such property, which gives lowest frequency of oscillation as small as possible. For $ab=1,\cdots,50$, we find that if $ab=15$, then the integral points on elliptic curve $E_{15}$ (up to negation of Mordel-Weil group) are 
$$
(-30 , 0 ),\, (-5 , 50 ),\, (0 , 0 ), \, (15 ,0 ),\, (24 , 108 ),\, (90 , 900).
$$
The only point $(x,y)$ with $ab|x$ and $ab^2|y$ is (90,900), which gives $n_0=6$, and the lowest frequency $\Lambda(n_0)=\sqrt{n_0(n_0-1)(n_0+2)}$ of oscillation is $4\sqrt{15}\simeq15.49\cdots$. 

There are other choices of $a,b$. We list down the value of $ab$ below 50, the corresponding $n_0$ and the lowest frequency $\Lambda(n_0)$:
\begin{center}
    \begin{tabular}{lll}
        $ab$ & $n_0$ & $\Lambda(n_0)$ \\
        15 & 6 & $4\sqrt{15}$ \\
        17 & 49 & $4\sqrt{323}$ \\
        22 & 9 & $6\sqrt{22}$ \\
        26 & 50 & $15\sqrt{78}$ \\
        42 & 7 & $3\sqrt{42}$ \\
        46 & 576 & $2040\sqrt{46}$ \\
        50 & 25 & $90\sqrt{2}$
    \end{tabular}
\end{center}

See Appendix \ref{A} for the MAGMA code used to find these values. This list suggests that $n_0=6$ might be the smallest order that meets the requirement, but this remains unproved. We summarize these into the following number-theoretic question:

\begin{question}\label{Q2}
For the family of elliptic curves
$$
E_{ab}:\,y^2=x(x-ab)(x+2ab),\quad a,b\in\mathbb{N},
$$
how many choices of $a,b\in\mathbb{N}$ are there such that, there is exactly one integral point $(x,y)\in E_{ab}$ with $x,y>0$ and $ab | x$, $a^2b | y$? For such $a,b$ and integral point $(x,y)$, is the minimal value of $x/(ab)$ exactly 6, or is it smaller?
\end{question}

Of course, a complete answer of Question \ref{Q2} should imply very clear understanding of periodic solutions of the spherical capillary water waves equation constructed using bifurcation analysis. But at this moment we are satisfied with existence, so we may pick any $\omega_0=a/b$ and $(j_0,n_0)$ that meets the requirement, for example the simplest case $n_0=6$, and any of the following choices of $\omega_0$ and $j_0$:
$$
\omega_0=\sqrt{15},\,j_0=4;
\quad
\omega_0=\sqrt{\frac{1}{15}},\,j_0=60;
\quad
\omega_0=\sqrt{\frac{3}{5}},\,j_0=20;
\quad
\omega_0=\sqrt{\frac{5}{3}},\,j_0=12.
$$
We thus refine our conjecture as follows:

\begin{conjecture}\label{Conj1}
Let $(n_0,j_0)$ be a pair of natural numbers with $n_0\geq2$, $j_0\geq1$, and set $\omega_0=\sqrt{\Lambda(n_0)}/j_0$. Suppose that the only natural number solution of the Diophantine equation
$$
\omega_0^2j_0=\Lambda(n_0)^2=n_0(n_0-1)(n_0+1)
$$
is $(j_0,n_0)$. Then there is a Cantor set with positive measure of parameters $\omega$, clustered near $\omega_0$, such that the spherical capillary water waves equation (\ref{Periodic}) admits small amplitude periodic solution with frequency $\omega$.
\end{conjecture}

The counterpart of Conjecture \ref{Conj1} for gravity capillary standing water waves was proved by Alazard-Baldi \cite{AB2015} using a Nash-Moser type theorem. The key technique in their proof was to find a conjugation of the linearized operator of the gravity capillary water waves system on $\mathbb{T}^1$ to an operator of the form
$$
\omega\partial_t+iT+i\lambda_1|D_x|^{1/2}+i\lambda_{-1}|D_x|^{-1/2}+\text{Operator of order}\leq-\frac{3}{2},
$$
where $T$ is an elliptic Fourier multiplier of order $3/2$, and $\lambda_1,\lambda_{-1}$ are real constants. The frequency $\omega$ lives in a Cantor type set that clusters around a given frequency so that the kernel of the linearized operator is 1-dimensional. With this conjugation, they were able to find periodic solutions of linearized problems required by Nash-Moser iteration.

It is expected that this technique could be transplanted to the equation (\ref{Periodic}), since our analysis for (\ref{EllipticCurve}) suggests that the 1-dimensional kernel requirement for bifurcation analysis is met. It seems that the greatest technical issue is to find a suitable conjugation that takes the linearized operator of (\ref{Periodic}) to an operator of the form  
$$
\omega\partial_t+i(T_{3/2}+T_{1/2}+T_{-1/2})+\text{Operator of order}\leq-\frac{3}{2},
$$
where each $T_k$ is a real Fourier multiplier acting on spherical harmonics. The difficulty is that, since we are working with pseudo-differential operators on $S^2$, the formulas of symbolic calculus are not as neat as those on flat spaces. It seems necessary to implement some global harmonic analysis for compact homogeneous spaces, e.g. extension of results collected in Ruzhansky-Turunen \cite{RT2009}. Unfortunately, those results do not include pseudo-differential operators with ``rough coefficients" and para-differential operators, so it seems necessary to re-write the whole theory.

\subsection{Number-Theoretic Obstruction with Normal Form Reduction}
As pointed out in Section 1, the system (\ref{EQFormal}) is locally well-posed due to a result in \cite{BeGu1998}. General well-posedness results \cite{CoSh2007}, \cite{ShZe2008} for free boundary value problem of Euler equation also apply. As for lifespan estimate for initial data $\varepsilon$-close to the static solution (\ref{Stat}), it should not be hard to conclude that the lifespan should be bounded below by $1/\varepsilon$. The result relates to the fact that the sphere is a \emph{stable} critical point of the area functional, cf. \cite{BD2012}. This is nothing new: a suitable energy inequality should imply it. However, although the clue is clear, the implementation is far from standard since we are working on a compact manifold. A rigorous proof still calls for hard technicalities. 

Now we will be looking into the nonlinear equation (\ref{EQFormal}) for its longer time behavior. Although appearing similar to the well-studied water waves equation in e.g. \cite{IoPu2019}, \cite{Lannes2005}, \cite{Wu1997}, \cite{Wu1999}, there is a crucial difference between the dispersive relation in (\ref{EQFormal}) and the well-studied water waves equations: the dispersive relation exhibits a strong \emph{rigidity property}, i.e. the arbitrary physical constants enter into the dispersive relation $\Lambda$ only as \emph{scaling factors}. For the gravity-capillary water waves, the linear dispersive relation reads
$$
\sqrt{g|\nabla|+\sigma|\nabla|^3},
$$
where $g$ is the gravitational constant and $\sigma$ is the surface tension coefficient. For the Klein-Gordon equation on a Riemannian manifold, the linear dispersive relation reads
$$
\sqrt{-\Delta+m^2},
$$
where $m$ is the mass. In \cite{IoPu2019}, Ionescu and Pusateri referred such dispersive relations as having \emph{non-degenerate dependence on physical parameter}, while in our context it is appropriate to refer to the dependence as \emph{degenerate}. We will see that this crucial difference brings about severe obstructions for the long-time well-posedness of the system.

Following the idea of Delort and Szeftel \cite{DS2004}, we look for a normal form reduction of (\ref{EQFormal}) and explain why the rigidity property could cause obstructions. Not surprisingly, the obstruction is due to resonances, and strongly relates to the solvablity of a Diophantine equation. Delort and Szeftel cast a normal form reduction to the small-initial-data problem of quasilinear Klein-Gordon equation on the sphere and obtained an estimate on the lifespan longer than the one provided by standard energy method. After their work, normal form reduction has been used by mathematicians to understand water waves on flat tori, for example \cite{BD2018}, \cite{BFP2018} and \cite{IoPu2019}. The idea was inspired by the normal form reduction method introduced by Shatah \cite{Shatah1985}: for a quadratic perturbation of a linear dispersive equation
$$
\partial_tu+iLu=N(u)=O(u^{\otimes 2}),
$$
using a new variable $u+B(u,\bar u)$ with a suitably chosen quadratic addendum $B(u,\bar u)$ can possibly eliminate the quadratic part of $N$, thus extending the lifespan estimate beyond the standard $1/\varepsilon$.

So we shall write the quadraticr part of $\mathfrak{N}(u)$ as
$$
\sum_{n_3\geq0}\Pi^{(n_3)}\left[\sum_{n_1\geq0}\sum_{n_2\geq0}
\mathcal{M}_1\left(
\Pi^{(n_1)}u,\Pi^{(n_2)}u
\right)+\mathcal{M}_2\left(
\Pi^{(n_1)}u,\Pi^{(n_2)}\bar u
\right)+\mathcal{M}_3\left(
\Pi^{(n_1)}\bar u,\Pi^{(n_2)}\bar u
\right)\right],
$$
where $\mathcal{M}_{1},\mathcal{M}_{2},\mathcal{M}_{3}$ are complex bi-linear operators, following the argument of Section 4 in \cite{DS2004}. They are independent of $t$ since the right-hand-side of the equation does not depend on $t$ explicitly. Let's look for a diffeomorphism
$$
u\to v:=u+\mathbf{B}[u,u]
$$
in the function space $C^\infty(S^2)$, where $\mathbf{B}$ is a bilinear operator, so that the equation (\ref{EQFormal}) with quadratic nonlinearity reduces to an equation with cubic nonlinearity. The $\mathbf{B}[u,u]$ is supposed to take the form $\mathbf{B}[u,u]=\mathbf{B}_1[u,u]+\mathbf{B}_2[u,u]+\mathbf{B}_3[u,u]$, with
$$
\mathbf{B}_1[u,u]
=\sum_{n_3\geq0}\sum_{n_1,n_2\geq0}b_1(n_1,n_2,n_3)\Pi^{(n_3)}\mathcal{M}_1\left(
\Pi^{(n_1)}u,\Pi^{(n_2)}u
\right),
$$
$$
\mathbf{B}_2[u,u]
=\sum_{n_3\geq0}\sum_{n_1,n_2\geq0}b_2(n_1,n_2,n_3)\Pi^{(n_3)}\mathcal{M}_2\left(
\Pi^{(n_1)}u,\Pi^{(n_2)}\bar u
\right),
$$
$$
\mathbf{B}_3[u,u]
=\sum_{n_3\geq0}\sum_{n_1,n_2\geq0}b_3(n_1,n_2,n_3)\Pi^{(n_3)}\mathcal{M}_3\left(
\Pi^{(n_1)}\bar u,\Pi^{(n_2)}\bar u
\right),
$$
where the $b_j(n_1,n_2,n_3)$'s are complex numbers to be determined. Implementing (\ref{EQFormal}), we find
$$
\begin{aligned}
(\partial_t&+i\Lambda)(u+\mathbf{B}[u,u])\\
&=\mathfrak{N}(u)+\sum_{n_3\geq0}\sum_{n_1,n_2\geq0}b_1(n_1,n_2,n_3)\Pi^{(n_3)}\mathcal{M}_1\left(
\Pi^{(n_1)}\partial_t u,\Pi^{(n_2)} u\right)\\
&\quad+\sum_{n_3\geq0}\sum_{n_1,n_2\geq0}b_1(n_1,n_2,n_3)\Pi^{(n_3)}\mathcal{M}_1\left(
\Pi^{(n_1)} u,\Pi^{(n_2)}\partial_t u\right)+(\text{similar terms})\\
&\quad+\sum_{n_3\geq0}\sum_{n_1,n_2\geq0}i\Lambda(n_3)b_1(n_1,n_2,n_3)\Pi^{(n_3)}\mathcal{M}_1\left(
\Pi^{(n_1)}u,\Pi^{(n_2)}u\right)\\
&=\mathfrak{N}(u)+\sum_{n_3\geq0}\sum_{\min(n_1,n_2)\leq1}i\left[\Lambda(n_3)-\Lambda(n_1)-\Lambda(n_2)\right]
b_1(n_1,n_2,n_3)\Pi^{(n_3)}\mathcal{M}_1\left(
\Pi^{(n_1)} u,\Pi^{(n_2)} u\right)\\
&\quad+\sum_{n_3\geq0}\sum_{n_1,n_2\geq2}i\left[\Lambda(n_3)-\Lambda(n_1)-\Lambda(n_2)\right]
b_1(n_1,n_2,n_3)\Pi^{(n_3)}\mathcal{M}_1\left(
\Pi^{(n_1)} u,\Pi^{(n_2)} u\right)\\
&\quad+(\text{similar terms})+O(u^{\otimes3}).\\
\end{aligned}
$$
We aim to eliminate most of the second order portions of $\mathfrak{N}(u)$. The coefficients $b_{j}(n_1,n_2,n_3)$ are fixed as follows:
\begin{equation}\label{B123}
\begin{aligned}
b_{1}(n_1,n_2,n_3)&=i\left[\Lambda(n_3)-\Lambda(n_1)-\Lambda(n_2)\right]^{-1},\quad n_1,n_2,n_3\geq2\\
b_{2}(n_1,n_2,n_3)&=i\left[\Lambda(n_3)-\Lambda(n_1)+\Lambda(n_2)\right]^{-1},\quad n_1,n_2,n_3\geq2\\
b_{3}(n_1,n_2,n_3)&=i\left[\Lambda(n_3)+\Lambda(n_1)+\Lambda(n_2)\right]^{-1},\quad n_1,n_2,n_3\geq2\\
b_{1,2,3}(n_1,n_2,n_3)&=0,\quad\text{if }\Lambda(n_3)\pm\Lambda(n_1)\pm\Lambda(n_2)=0\text{ or }\min(n_1,n_2,n_3)\leq1,
\end{aligned}
\end{equation}
then a large portion of the second order part of $\Pi_{\geq2}\mathfrak{N}\left(u\right)$ will be eliminated. In fact, for $n_1,n_2,n_3\geq2$, if $\Lambda(n_3)\pm\Lambda(n_1)\pm\Lambda(n_2)\neq0$, then the term
$$
\Pi^{(n_3)}\left[\sum_{n_1,n_2\geq2}\mathcal{M}_1\left(
\Pi^{(n_1)}u,\Pi^{(n_2)}u
\right)+\mathcal{M}_2\left(
\Pi^{(n_1)}u,\Pi^{(n_2)}\bar u
\right)+\mathcal{M}_3\left(
\Pi^{(n_1)}\bar u,\Pi^{(n_2)}\bar u
\right)\right]
$$
is cancelled out. On the other hand, by the volume preservation equality (\ref{Vol}) and conservation law (\ref{SphereCons.}), there holds $\Pi^{(0)}u=O(u^{\otimes2})$, $\Pi^{(1)}u=O(u^{\otimes2})$, so the low-low interaction
$$
\Pi_{\geq2}\left[\sum_{\min(n_1,n_2)\leq1}\mathcal{M}_1\left(
\Pi^{(n_1)}u,\Pi^{(n_2)}u
\right)+\mathcal{M}_2\left(
\Pi^{(n_1)}u,\Pi^{(n_2)}\bar u
\right)+\mathcal{M}_3\left(
\Pi^{(n_1)}\bar u,\Pi^{(n_2)}\bar u
\right)\right]
$$
is automatically $O(u^{\otimes3})$.

Thus the existence and continuity of the normal form $\mathbf{B}[u,u]$ depends on the property of the 3-way resonance equation
\begin{equation}\label{3wayRes}
\Lambda(n_3)-\Lambda(n_1)-\Lambda(n_2)=0,\quad n_1,n_2,n_3\geq2.
\quad
n_3\leq n_1+n_2
\end{equation}
which is equivalent to the Diophantine equation
\begin{equation}\label{Diophantine}
[F(n_1)+F(n_2)-F(n_3)]^2-4F(n_1)F(n_2)=0,\quad n_1,n_2,n_3\geq2,
\end{equation}
where $F(X)=X(X-1)(X+2)$. 

If the tuple $(n_1,n_2,n_3)$ is non-resonant, i.e. it is such that $b_{1,2,3}(n_1,n_2,n_3)\neq0$, then some elementary number theoretic argument will give a lower bound on $|b_{1,2,3}(n_1,n_2,n_3)|$ in terms of a negative power (can be fixed as $-9/2$) of $n_1,n_2,n_3$. This is usually referred as \emph{small divisor estimate}.

To study the distribution of resonant frequencies, we propose the following unsolved question:

\begin{question}\label{Q3}
Does the Diophantine equation (\ref{Diophantine}) have finitely many solutions?
\end{question}

However, the Diophantine equation (\ref{Diophantine}) does admit non-trivial solutions $(5,5,8)$ and $(10,10,16)$. In other words, the second order terms e.g.
$$
\Pi^{(8)}\mathcal{M}_1(\Pi^{(5)}u\cdot\Pi^{(5)}u),
\quad
\Pi^{(5)}\mathcal{M}_1(\Pi^{(8)}u\cdot\Pi^{(5)}\bar u),
$$
in the quadratic part of $\mathfrak{N}(u)$ cannot be eliminated by normal form reduction. On the other hand, it seems to be very hard to determine whether (\ref{Diophantine}) still admits any other solution. We have the following proposition (the author would like to thank Professor Bjorn Poonen for the proof):
\begin{proposition}
The Diophantine equation (\ref{Diophantine}) has no solution with $n_1\leq10^4$ other than $(5,5,8)$ and $(10,10,16)$.
\end{proposition}
The proof of this proposition is computer-aided. The key point is to use the so-called Runge's method to show that if $(n_1,n_2,n_3)$ is a solution, then there must hold $n_2=O(n_1^2)$. For a given $n_1$, this reduces the proof to numerical verification for \emph{finitely many possibilities}. The algorithm can of course be further optimized, but due to some algebraic geometric considerations, it is reasonable to conjecture that the solution of (\ref{Diophantine}) should be very rare. In fact, there are two ways of viewing the problem. We observe that if $(n_1,n_2,n_3)$ is a solution, then $F(n_1)F(n_2)$ must be a square, and the square free part of $F(n_1),F(n_2),F(n_3)$ must be the same. Further reduction turns the problem into finding integral points on a family of elliptic curves
$$
Y^2=cF(X),\quad c\text{ is square-free},
$$
which is of course difficult, but since Siegel's theorem asserts that there are only finitely many integer points on an elliptic curve over $\mathbb{Q}$, it is reasonable to conjecture that there are not ``too many" solutions to (\ref{Diophantine}). We may also view the problem as finding integral (rational) points on a given algebraic surface. The complex projective surface corresponding to (\ref{Diophantine}) is given by
$$
\begin{aligned}
\mathfrak{V}:[X(X-W)(X+2W)&+Y(Y-W)(Y+2W)-Z(Z-W)(Z+2W)]^2\\
&\quad=4X(X-W)(X+2W)Y(Y-W)(Y+2W),
\end{aligned}
$$
where $[X,Y,Z,W]$ is the homogeneous coordinate on $\mathbb{CP}^3$. With the aid of computer, we obtain
\begin{proposition}
The complex projective surface $\mathfrak{V}\subset\mathbb{CP}^3$ has Kodaira dimension 2 (i.e. it is of general type under Kodaira-Enriquez classification), and its first Betti number is 0.
\end{proposition}
See Appendix \ref{A} for the code.

The first part of the proposition suggests that the rational points of $\mathfrak{V}$ should be localized on finitely many algebraic curves laying on $\mathfrak{V}$; nevertheless, this seemingly simple suggestion is indeed a special case of the \emph{Bomberi-Lang conjecture}, a hard problem in number theory (its planar case is known as the celebrated Faltings's theorem). The second part suggests that the rational points of $\mathfrak{V}$ should be rare since its Albanese variety is a single point. But these are just heuristics that solutions to (\ref{Diophantine}) should be rare. In general, determining the solvability of a given Diophantine equation is very difficult \footnote{For example, the seemingly simple Diophantine equation $x^3+y^3+z^3=42$ is in fact a puzzle of more than 60 years, and its first solution was found recently by Booker-Sutherland \cite{BS2021}. It is of extremely large magnitude: $42=(-80\ 538\ 738\ 812\ 075\ 974)^{3}+80\ 435\ 758\ 145\ 817\ 515^{3}+12\ 602\ 123\ 297\ 335\ 631^{3}$. Another example is the equation of same type $x^3+y^3+z^3=3$. Beyond the easily found solutions (1, 1, 1), (4, 4, -5), (4, -5, 4), (-5, 4, 4), the next solution reads $(569\ 936\ 821\ 221\ 962\ 380\ 720, -569\ 936\ 821\ 113\ 563\ 493\ 509, -472\ 715\ 493\ 453\ 327\ 032)$.}, as number theorists and arithmetic geometers generally believe.

The reason that such issues do not occur for water waves in the flat setting or nonlinear Klein-Gordon equations is twofold. First of all, the resonance equation is easily understood even in the degenerate case in the flat setting. For example, the capillary water waves without gravity on $\mathbb{T}^2$ has dispersive relation $|\nabla|^{1/2}$, and the 3-way resonance equation is
$$
\sqrt[4]{k_1^2+k_2^2}+\sqrt[4]{l_1^2+l_2^2}=\sqrt[4]{m_1^2+m_2^2},
$$
with the additional requirement $m=k+l$. We already know that the resonance equation has no non-trivial solution at all, cf. \cite{BFP2018}. But even without using $m=k+l$, we would be able to conclude that there are at most finitely many non-trivial solutions from the celebrated Faltings's theorem on rational points on high-genus algebraic projective curves (although this is like using a sledge hammer to crack a nut). Secondly, for the non-degenrate case, for example the gravity-capillary waves, the dispersive relation reads $\sqrt{g|\nabla|+\sigma|\nabla|^3}$, so if the ratio $\sigma/g$ is a transcedental number then the 3-way resonance equation has no solution. Furthermore, using some elementary calculus and a measure-theoretic argument, it can be shown, not without technicalities, that the resonances admit certain small-divisor estimates for almost all parameters. This is exactly the argument employed by Delort-Szeftle \cite{DS2004}, Berti-Delort \cite{BD2018} and Ionescu-Pusateri \cite{IoPu2019}, so that their results were stated for \emph{almost all} parameters. These parameters are, roughly speaking, badly approximated by algebraic numbers. 

However, the resonance equation (\ref{3wayRes}) is inhomogeneous and allows no arbitrary physical parameter at all. Furthermore, since product of spherical harmonics are no longer spherical harmonics in general, Fourier series techniques employed by \cite{BFP2018} \cite{BD2018} \cite{IoPu2019} that works for the torus are never valid for $S^2$; for example, we cannot simply assume $n_3=n_1+n_2$ in (\ref{3wayRes}), as already illustrated by the solutions (5,5,8) (10,10,16). These are the crucial differences between the capillary spherical water waves and all known results for water waves in the flat setting.

\subsection{Heuristics for Lifespan Estimate}
To summarize, almost global lifespan estimate of (\ref{EQFormal}) depends on the difficult number theoretic question \ref{Q3}. Before it is fully resolved, we can only expect partial results regarding the normal form transformation. 

If there are only finitely many solutions to the Diophantine equation (\ref{Diophantine}), then under the normal form reduction $u\to v=u+\mathbf{B}[u,u]$ with coefficients given by (\ref{B123}), the equation (\ref{EQFormal}) is transformed into the following system:
$$
\begin{aligned}
\frac{\partial }{\partial t}\Pi_cv&=O(v^{\otimes 2}),\\
\frac{\partial }{\partial t}(1-\Pi_c)v&=O(v^{\otimes 3}),
\end{aligned}
$$
where $\Pi_c$ is the orthogonal projection to $\bigoplus_{n_3}\mathcal{E}^{(n_3)}\subset L^2(S^2)$, with $n_3$ being either 0 or 1, or exhausting the third component of all nontrivial solutions of (\ref{Diophantine}). 

There is no reasonable assertion to be made if the conjecture fails. However, if the conjecture does hold true, then we can expect that the lifespan estimate for $\varepsilon$-Cauchy data of (\ref{EQ}) goes beyond $\varepsilon^{-1}$, as what we expect for gravity water waves in the periodic setting, e.g. in \cite{IoPu2019}:
\begin{conjecture}\label{Conj2}
If the Diophantine equation (\ref{Diophantine}) has only finitely many solutions, then there is some $\alpha>0$ such that for $\varepsilon$-Cauchy data of (\ref{EQ}), the lifespan goes beyond $\varepsilon^{-(1+\alpha)}$ as $\varepsilon\to0$.
\end{conjecture}

Let's explain the heuristic as follows. The argument we aim to implement is the standard ``continuous induction method", i.e. for some suitably large $s$ and $K$ and suitable $\alpha>0$, assuming $T=\varepsilon^{-(1+\alpha)}$ and $\sup_{t\in[0,T]}\|v\|_{H^s(g_0)}\leq K\varepsilon$, we try to prove a better bound $\sup_{t\in[0,T]}\|v\|_{H^s(g_0)}\leq K\varepsilon/2$. Here $g_0$ is the standard metric on $S^2$. It is intuitive to expect such a result for the cubic equation ${\partial_t}(1-\Pi_c)v=O(v^{\otimes 3})$. As for the quadratic equation ${\partial_t}\Pi_cv=O(v^{\otimes 2})$, it is crucial to implement the conservation of energy $\mathbf{H}[\zeta,\phi]\equiv4\pi$ for a solution. We summarize it as 
\begin{proposition}\label{Pi_c}
Fix $T>0$. Let $u$ be a smooth solution of (\ref{EQFormal}) and $v=u+\mathbf{B}[u,u]$ be as above. Suppose for some suitably large $s$ and $K$, there holds
$$
\sup_{t\in[0,T]}\|v\|_{H^s(g_0)}\leq K\varepsilon
$$
with $\varepsilon$ sufficiently small. Then there is in fact a better bound for the low frequency part $\Pi_cv$:
$$
\sup_{t\in[0,T]}\|\Pi_cv\|_{H^s(g_0)}\leq K\varepsilon/4.
$$
\end{proposition}
\begin{proof}
We consider the ``approximate" energy functional
$$
\mathbf{H}_0[\zeta,\phi]=4\pi+\int_{S^2}2\zeta\cdot d\mu_0+\frac{1}{2}\int_{S^2}(|\nabla_0\zeta|^2+2|\zeta|^2)d\mu_0+\frac{1}{2}\int_{S^2}\left||\nabla_0^{1/2}|\phi\right|^2 d\mu_0,
$$
where $\mu_0$ is the standard area measure on $S^2$ and $\nabla_0$ is the standard connection on $S^2$. This is nothing but the quadratic approximation to $\mathbf{H}[\zeta,\phi]$ in (\ref{SHamilton}) at $(0,0)$, so there holds $\mathbf{H}_0[\zeta,\phi]=\mathbf{H}[\zeta,\phi]+O(u^{\otimes3})$. Using volume preservation (\ref{Vol}), we obtain $\int_{S^2}\zeta d\mu_0=-\|\zeta\|_{L^2(g_0)}^2+O(\zeta^{\otimes3})$, so summarizing we have
\begin{equation}\label{O_3}
\frac{1}{2}\int_{S^2}(|\nabla_0\zeta|^2-2|\zeta|^2)d\mu_0+\frac{1}{2}\int_{S^2}\left||\nabla_0^{1/2}|\phi\right|^2 d\mu_0
=O(u^{\otimes3}).
\end{equation}
Note that we used the conservation law $\mathbf{H}[\zeta,\phi]\equiv4\pi$. By spectral calculus on $S^2$, we have
$$
\|\zeta\|_{H^1(g_0)}^2
\simeq\|\Pi^{(0)}\zeta\|_{L^2(g_0)}^2+\|\Pi^{(1)}\zeta\|_{L^2(g_0)}^2+\int_{S^2}(|\nabla_0\zeta|^2-2|\zeta|^2)d\mu_0,
$$
and by volume preservation and conservation of momentum (\ref{SphereCons.}) we find
\begin{equation}\label{zetaeq}
\|\zeta\|_{H^1(g_0)}^2
\simeq\int_{S^2}(|\nabla_0\zeta|^2-2|\zeta|^2)d\mu_0
+O(\zeta^{\otimes4}).
\end{equation}

Now, for some $N_0>0$ relating to the loss of regularity caused by $\mathbf{B}$, we may choose $s>>2N_0$. Then if $\|v\|_{H^s(g_0)}\leq K\varepsilon$, it follows that $\|u\|_{H^{s-N_0}(g_0)}\leq K'\varepsilon$, so by (\ref{O_3}) and (\ref{zetaeq}) we have
$$
\|u\|_{L^2(g_0)}^2\simeq\|\zeta\|_{H^1(g_0)}^2+\|\phi\|_{H^{1/2}(g_0)}^2\leq K'\varepsilon^3.
$$
Thus
$$
\begin{aligned}
\|v\|_{L^2(g_0)}
&\leq C\|u\|_{L^2(g_0)}+C\|u\|_{H^{N_0}(g_0)}^2\\
&\leq C\varepsilon^{3/2}+K'\varepsilon^2.
\end{aligned}
$$
Since the spectrum of $\Pi_cv$ is bounded, by Bernstein type inequality we have 
$$
\|\Pi_cv\|_{H^s(g_0)}
\leq C\|v\|_{L^2(g_0)}
\leq K'\varepsilon^{3/2}(1+\varepsilon^{1/2}).
$$
If $\varepsilon$ is sufficiently small then this implies $\|\Pi_cv\|_{H^s(g_0)}\leq K\varepsilon/4$.
\end{proof}
We point out that the above proof is independent from the magnitude of the lifespan $T$, so it is always applicable as long as the cubic equation ${\partial_t}(1-\Pi_c)v=O(v^{\otimes 3})$ is well-understood. There are two crucial points in the proof of Proposition \ref{Pi_c}: the conservation of energy, and that the projection $\Pi_c$ is of finite rank, so that a Bernstein type inequality holds. The last fact holds only if there are finitely many 3-way resonances, i.e. there are only finitely many solutions to the Diophantine equation (\ref{3wayRes}).

Finally, we propose an even more ambitious conjecture concerning global dynamical properties of spherical water droplets, which is again illuminated by observation in hydrodynamical experiments under zero gravity, and also suggested by the results of Berti-Montalto \cite{BM2020}:

\begin{conjecture}\label{Conj3}
If the Diophantine equation (\ref{Diophantine}) has only finitely many solutions, then a KAM type result holds for (\ref{EQ}): there is a family of infinitely many quasi-periodic solutions of (\ref{EQ}), depending on a parameter which takes value in a Cantor-type set.
\end{conjecture}

\appendix
\section{MAGMA Code}\label{A}
MAGMA is a large, well-supported software package designed for computations in algebra, number theory, algebraic geometry, and algebraic combinatorics. In this appendix, we give the MAGMA code used to conduct computations on Diophantine equations related to the spherical capillary water waves system. 

\subsection{Integral Points on Elliptic Curve} We can find all integral points on a given elliptic curve over $\mathbb{Q}$ using MAGMA. For a monic cubic polynomial $f(x)$, the function \verb|EllipticCurve(f)| creates the elliptic curve
$$
E:y^2=f(x),
$$
and the function \verb|IntegralPoints(E)| returns a sequence containing all the integral points on $E$ under the homogeneous coordinate of $\mathbb{QP}^2$, modulo negation. We use this to find out all integral points on the elliptic curve 
$$
E_c:\,y^2=x^3+cx^2-2c^2x=x(x-c)(x+2c).
$$
for natural number $c\leq50$. The MAGMA code is listed below, which excludes all the $c$'s such that there are only trivial integral points $\{(-c,0),(0,0),(c,0)\}$ on $E_c$.

\begin{spverbatim}
> Qx<x> := PolynomialRing(Rationals());
> for c in [1..50] do
> E := EllipticCurve(x^3+c*x^2-2*c^2*x);
> S, reps := IntegralPoints(E);
> if # S gt 3 then
> print c, E;
> print S;
> end if;
> end for;
2 Elliptic Curve defined by y^2 = x^3 + 2*x^2 - 8*x over Rational Field
[ (-4 : 0 : 1), (-2 : 4 : 1), (-1 : -3 : 1), (0 : 0 : 1), (2 : 0 : 1), (4 : 8 : 1), (8 : -24 : 1), (50 : 360 : 1) ]
8 Elliptic Curve defined by y^2 = x^3 + 8*x^2 - 128*x over Rational Field
[ (-16 : 0 : 1), (-8 : 32 : 1), (-4 : -24 : 1), (0 : 0 : 1), (8 : 0 : 1), (9 : -15 : 1), (16 : 64 : 1), (32 : -192 : 1), (200 : 2880 : 1) ]
13 Elliptic Curve defined by y^2 = x^3 + 13*x^2 - 338*x over Rational Field
[ (-26 : 0 : 1), (0 : 0 : 1), (13 : 0 : 1), (121 : 1386 : 1) ]
15 Elliptic Curve defined by y^2 = x^3 + 15*x^2 - 450*x over Rational Field
[ (-30 : 0 : 1), (-5 : -50 : 1), (0 : 0 : 1), (15 : 0 : 1), (24 : 108 : 1), (90 : -900 : 1) ]
17 Elliptic Curve defined by y^2 = x^3 + 17*x^2 - 578*x over Rational Field
[ (-34 : 0 : 1), (-32 : -56 : 1), (0 : 0 : 1), (17 : 0 : 1), (833 : 24276 : 1) ]
18 Elliptic Curve defined by y^2 = x^3 + 18*x^2 - 648*x over Rational Field
[ (-36 : 0 : 1), (-32 : -80 : 1), (-18 : 108 : 1), (-9 : -81 : 1), (0 : 0 : 1), (18 : 0 : 1), (36 : 216 : 1), (72 : -648 : 1), (450 : 9720 : 1) ]
22 Elliptic Curve defined by y^2 = x^3 + 22*x^2 - 968*x over Rational Field
[ (-44 : 0 : 1), (-32 : -144 : 1), (0 : 0 : 1), (22 : 0 : 1), (198 : 2904 : 1) ]
23 Elliptic Curve defined by y^2 = x^3 + 23*x^2 - 1058*x over Rational Field
[ (-46 : 0 : 1), (0 : 0 : 1), (23 : 0 : 1), (50 : -360 : 1) ]
26 Elliptic Curve defined by y^2 = x^3 + 26*x^2 - 1352*x over Rational Field
[ (-52 : 0 : 1), (-49 : -105 : 1), (0 : 0 : 1), (26 : 0 : 1), (1300 : 47320 : 1) ]
30 Elliptic Curve defined by y^2 = x^3 + 30*x^2 - 1800*x over Rational Field
[ (-60 : 0 : 1), (-50 : 200 : 1), (-45 : -225 : 1), (-24 : -216 : 1), (-20 : 200 : 1), (-6 : 108 : 1), (0 : 0 : 1), (30 : 0 : 1), (36 : 144 : 1), (40 : -200 :
1), (75 : -675 : 1), (90 : 900 : 1), (300 : 5400 : 1), (324 : -6048 : 1), (480 : -10800 : 1), (7290 : 623700 : 1), (10830 : -1128600 : 1), (226875 : 108070875 : 1) ]
32 Elliptic Curve defined by y^2 = x^3 + 32*x^2 - 2048*x over Rational Field
[ (-64 : 0 : 1), (-32 : 256 : 1), (-16 : -192 : 1), (0 : 0 : 1), (32 : 0 : 1), (36 : -120 : 1), (64 : 512 : 1), (128 : -1536 : 1), (800 : 23040 : 1) ]
33 Elliptic Curve defined by y^2 = x^3 + 33*x^2 - 2178*x over Rational Field
[ (-66 : 0 : 1), (0 : 0 : 1), (33 : 0 : 1), (81 : -756 : 1) ]
35 Elliptic Curve defined by y^2 = x^3 + 35*x^2 - 2450*x over Rational Field
[ (-70 : 0 : 1), (-49 : 294 : 1), (-45 : -300 : 1), (-40 : 300 : 1), (-14 : -196 : 1), (0 : 0 : 1), (35 : 0 : 1), (50 : 300 : 1), (175 : -2450 : 1), (224 : 3528 : 1), (280 : -4900 : 1), (4410 : -294000 : 1), (14450 : -1739100 : 1) ]
39 Elliptic Curve defined by y^2 = x^3 + 39*x^2 - 3042*x over Rational Field
[ (-78 : 0 : 1), (0 : 0 : 1), (39 : 0 : 1), (147 : -1890 : 1) ]
42 Elliptic Curve defined by y^2 = x^3 + 42*x^2 - 3528*x over Rational Field
[ (-84 : 0 : 1), (-56 : -392 : 1), (-12 : 216 : 1), (0 : 0 : 1), (42 : 0 : 1), (63 : -441 : 1), (294 : 5292 : 1) ]
43 Elliptic Curve defined by y^2 = x^3 + 43*x^2 - 3698*x over Rational Field
[ (-86 : 0 : 1), (-32 : -360 : 1), (0 : 0 : 1), (43 : 0 : 1) ]
46 Elliptic Curve defined by y^2 = x^3 + 46*x^2 - 4232*x over Rational Field
[ (-92 : 0 : 1), (0 : 0 : 1), (46 : 0 : 1), (26496 : 4316640 : 1) ]
50 Elliptic Curve defined by y^2 = x^3 + 50*x^2 - 5000*x over Rational Field
[ (-100 : 0 : 1), (-50 : 500 : 1), (-25 : -375 : 1), (-4 : 144 : 1), (0 : 0 : 1), (50 : 0 : 1), (100 : 1000 : 1), (200 : -3000 : 1), (1250 : 45000 : 1) ]
Running Magma V2.27-7.
Seed: 821911319; Total time: 2.430 seconds; Total memory usage: 85.16MB.
\end{spverbatim}

\subsection{Classification of Projective Surface} Some basic geometric parameters of the complex projective surface
$$
\begin{aligned}
\mathfrak{V}:[X(X-W)(X+2W)&+Y(Y-W)(Y+2W)-Z(Z-W)(Z+2W)]^2\\
&\quad=4X(X-W)(X+2W)Y(Y-W)(Y+2W)
\end{aligned}
$$
in $\mathbb{CP}^3$ can be computed using MAGMA. The author would like to thank Professor Bjorn Poonen for introducing MAGMA and providing the code listed below.

\begin{verbatim}
> Q:=Rationals();
> P<x,y,z,w>:=ProjectiveSpace(Q,3);
> fx:=x*(x-w)*(x+2*w);
> fy:=y*(y-w)*(y+2*w);
> fz:=z*(z-w)*(z+2*w);
> V:=Surface(P,(fz-fx-fy)^2-4*fx*fy);
> KodairaEnriquesType(V);
2 0 General type
Running Magma V2.27-7.
Seed: 1338492807; Total time: 0.670 seconds; Total memory usage: 32.09MB.
\end{verbatim}

Note that the Kodaira dimension is invariant regardless of the choice of base field, so it is legitimate to choose the base field to be $\mathbb{Q}$ in the above code. The variable $w$ is used to homogenize the equation. The function \verb|KodairaEnriquesType(V)| returns three values for the given projective surface $V$: the first is the Kodaira dimension, the second is irrelevant when the Kodaira dimension is not $-\infty$, 1 or 0, and the third is the Kodaira-Enriquez classification of the surface $X$.


\bibliographystyle{alpha}
\bibliography{References}

\end{spacing}
\end{document}